\documentclass[11pt]{amsart}
\usepackage[utf8]{inputenc}
 \usepackage[russian]{babel}
 \usepackage{amsmath,amssymb,url}

 \theoremstyle{plain}
  
  \newtheorem{lemma}{Лемма}
  \newtheorem{corollary}{Следствие}
  \newtoks\thehProclaim

 \theoremstyle{definition}

\theoremstyle{plain}
\newtoks\thehProclaim
\newtheorem*{Proclaim}{\the\thehProclaim}
\newenvironment{proclaim}[1]{\thehProclaim{#1}\begin{Proclaim}}{\end{Proclaim}}

\theoremstyle{definition}
\newtoks{\thehRemark}
\newtheorem*{Remark}{\the\thehRemark}

\begin{document}
  \begin{center}
    {\Large\bf О финитной отделимости конечно порожденных коммутативных колец}
 \end{center}
 \begin{center}\vspace{2mm}
     {\bf\small\copyright\ \ Станислав Кублановский}
 \end{center}\vspace{2mm}

    {\footnotesize   Установлены необходимые и достаточные условия финитной отделимости конечно порожденных коммутативных колец. Доказано, что каждое такое кольцо суть конечное расширение некоторого своего идеала кручения $I_{k}$ ($k$ число свободное от квадратов), который является подпрямым произведением конечного кольца и некоторого конечного набора колец без делителей нуля простых характеристик, являющихся целыми расширениями любого своего бесконечного моногенного подкольца.}

\address{ТПО « Северный Очаг»\\
Россия}

\email{stas1107@mail.ru}

\section{Введение}

Понятие финитной аппроксимируемости и финитной отделимости в алгебраических системах вызывает постоянный интерес исследователей. Одной из причин этого интереса является связь с алгоритмическими проблемами. На эту связь указал еще академик А.И. Мальцев в работе 1958 \cite{mal}. А именно: в финитно аппроксимируемых (финитно отделимых) конечно определенных системах разрешима проблема равенства (проблема вхождения). Все это относится и к основным алгебраическим системам: полугруппам, группам и кольцам.

Напомним, что алгебраическая система $A$ называется финитно отделимой, если для любого ее элемента $a$ и для любой подсистемы $ A^{'}$ такой, что $a\not\in A'$, существует конечная система $F$ и гомоморфизм $\varphi :A\rightarrow F$ такой, что $\varphi \left({a}\right)\not\in \varphi \left({A'}\right)$

Алгебраическая система $A$ называется финитно аппроксимируемой, если для любых ее двух различных элементов $a,b$  существует конечная система $F$ и гомоморфизм $\varphi :A\rightarrow F$ такой, что $\varphi \left({a}\right)\ne \varphi \left({b}\right)$.

Свойство финитной отделимости хорошо изучено в группах и полугруппах. Но для колец многие вопросы в этой тематике еще далеки от разрешения. Одним из таких вопросов был критерий финитной отделимости моногенных колец (напомним, что моногенным называют кольцо, порожденное одним элементом). Разрешение этого вопроса являлось бы отправной точкой для исследований. Оказалось, что в отличие от групп и полугрупп в кольцах дело обстоит значительно сложнее. Описание финитно отделимых моногенных колец получено автором в работе \cite{kyblv}. Данная работа является продолжением указанной работы. Целью исследования автора здесь является описание финитно отделимых конечно порожденных коммутативных колец. Основным результатом настоящей работы является следующая теорема.

\begin{proclaim}{Теорема} Для того, чтобы конечно порожденное коммутативное кольцо было финитно отделимо, необходимо и достаточно, чтобы оно было конечным расширением некоторого идеала кручения $I_{k}$ ($k$ --- число свободное от квадратов), который является  подпрямым произведением конечного кольца и некоторого конечного набора колец без делителей нуля простых характеристик, в которых любые два трансцендентные элемента целозависимы.
\end{proclaim}

\begin{proclaim}{Следствие}
Конечно порожденное коммутативное кольцо является финитно отделимым в том и только в том случае, когда каждое его двупорожденное подкольцо финитно отделимо.
\end{proclaim}

Из теоремы Гильберта о базисе следует, что каждое конечно порожденное коммутативное кольцо является конечно определенным, то есть может быть задано конечным набором определяющих соотношений. В настоящей работе приведен контрпример кольца простой характеристики без делителей нуля, показывающий, что из целой зависимости образующих кольца не следует, вообще говоря, целая зависимость произвольных двух трансцендентных элементов. Также в работе приведены примеры колец, в которых это имеет место. Поэтому представляет интерес описание условий финитной отделимости конечно порожденных коммутативных колец на языке образующих и определяющих соотношений.

В связи с особой ролью в этой проблематике двупорожденных колец простой характеристики автором доказано следующее

\begin{proclaim}{Предложение}
Двупорожденное кольцо простой характеристики $K=Z_{p}\left<{a,b\mid f\left({a,b}\right)=0}\right>$, где $f\left({x,y}\right)$ --- однородный многочлен от двух переменных, будет финитно отделимым тогда и только тогда, когда $f\left({x,y}\right)$ --- сепарабельный многочлен, то есть является произведением различных неприводимых многочленов.
\end{proclaim}

В §2 даны основные определения и обозначения, используемые в работе. В §3 приведены доказательства вспомогательных утверждений. В §4 изложено доказательство основных результатов настоящей работы. В §5 приведены пример и контрпример к естественным гипотезам, связанным с финитной отделимостью. В заключительном §6 сформулирован ряд открытых проблем по рассматриваемой тематике.

\section{Определения и обозначения}

Определение кольца и поля, идеала, подкольца, фактор-кольца и канонического гомоморфизма из кольца в фактор-кольцо, подпрямого произведения колец предполагаются известными. Рассматриваются кольца без требования существования 1, если не оговорено противное. Через $Z$ обозначается кольцо целых чисел. Если $p$ --- простое число, то через $Z_{p}$ обозначается кольцо остатков по модулю $p$ (оно, как известно, является простым полем). Каждое кольцо простой характеристики  $p$ рассматривается как алгебра над простым полем $Z_{p}$. Натуральное число $k$ называют свободным от квадратов, если оно не делится на квадрат простого числа. Легко видеть, что число свободное от квадратов это в точности произведение различных простых чисел, либо 1.

{\bf 1.} Через $Z_{p}[x]$ и $Z_{p}[x,y]$ обозначаются кольца многочленов с одной и двумя переменными с коэффициентами из простого поля $Z_{p}$.

{\bf 2.} Каждый многочлен $f\left({x,y}\right)$ из $Z_{p}[x,y]$ можно рассматривать как многочлен от одной переменой $x$ с коэффициентами из кольца $Z_{p}[y]$. Если старший коэффициент такого многочлена равен 1, то многочлен $f\left({x,y}\right)$ называется унитарным относительно $x$. Аналогично определяется многочлен унитарный относительно $y$. Многочлен $f\left({x,y}\right)$ называется унитарным, если он является унитарным по каждой переменной.

{\bf 3.}
  Однородный многочлен $f\left({x,y}\right)$ --- это многочлен, все одночлены которого имеют одинаковую полную степень.

{\bf 4.}
  Элемент $a$ кольца $K$ простой характеристики $p$ называется целым, если он в этом кольце является корнем некоторого ненулевого многочлена (без свободного члена) из $Z_{p}[x]$ (часто такие элементы называют алгебраическими над полем $Z_{p}$ ). Элемент, не являющийся целым, называется трансцендентным.

{\bf 5.}
   Элементы $a$ и $b$ кольца $K$ простой характеристики $p$ называются целозависимыми, если выполнено равенство $f\left({a,b}\right)=0$ для некоторого унитарного многочлена $f\left({x,y}\right)\in Z_{p}[x,y]$ без свободного члена.

{\bf 6.}
   Идеал $I$ кольца $K$ называется простым, если $I\ne K$ и для любых элементов $a,b$ кольца $K$ из $ab\in I$ следует, что $a\in I$ или $b\in I$. Легко видеть, что $I$ является простым идеалом тогда и только тогда, когда фактор-кольцо $K/I$ --- кольцо без делителей нуля.

{\bf 7.}
    Идеал $I$ коммутативного кольца $K$ называется примарным, если $I\ne K$ и для любых элементов $a,b$ кольца $K$ из $ab\in I$ следует, что $a\in I$ или $b^{n}\in I$ для некоторого натурального числа $n$.

{\bf 8.}
    Идеал $I$ кольца $K$ называется идеалом конечного индекса, если фактор кольцо $K/I$ конечно.

{\bf 9.}
    Кольцо $K$ называется конечным расширением своего идеала $I$, если аддитивная группа фактор-кольца $K/I$ является конечно порожденной.

{\bf 10.}
     Если фиксируется некоторое натуральное число $k$, то множество всех элементов $a$ кольца $K$, для которых выполнено равенство $k\cdot a=0$, образует идеал, который обозначают через $I_{k}$ и называют  идеалом $k$-кручения.

{\bf 11.}
    Мы будем использовать {\bf теорему Ласкера--Нетер}: {\it каждый идеал нётерова кольца можно записать в виде конечного пересечения примарных идеалов}. Доказательство этой теоремы можно найти в \cite{zarc}, §4.

{\bf 12.}
    Также будут использованы свойства полей и свойства алгебры многочленов над полем.Такие понятия как поле частных и алгебраическое замыкание поля \cite{byrbv}.

{\bf 13.}
 Будет использован результат А.И. Мальцева о замкнутости класса финитно отделимых колец относительно взятия конечных прямых произведений, подколец и гомоморфных образов \cite{mal}.

\section{Вспомогательные утверждения}
\begin{lemma}
 Если кольцо $K$ простой характеристики финитно отделимо  и $a^{2}=0$ для некоторого элемента $a\in K$, то для любого элемента $b\in K$ имеет место $a\cdot f\left({b}\right)=0$ для некоторого унитарного многочлена $f\left({x}\right)$ с целыми коэффициентами и без свободного члена.
\end{lemma}

\begin{proof}
Предположим противное, то есть пусть для некоторых элементов $a$ и $b$ в кольце $K$ простой характеристики $p$ имеет место $a^{2}=0$ и $a\cdot f\left({b}\right)\ne 0$ для любого унитарного многочлена $f\left({x}\right)$ с целыми коэффициентами и без свободного члена.

Рассмотрим множество элементов $a_{n}=a+a\left({b^{2n}-b^{n}}\right)$ и подкольцо $A$ кольца $K$, порожденное этим семейством элементов $M=\{a_{n}\mid n=1,2,\dots\}$. Нетрудно заметить, что $A$ совпадает с аддитивной подгруппой, порожденной семейством элементов $M$, то есть $A=Za_{1}+Za_{2}+ \dots + Za_{n}+ \dots$ (поскольку произведение любых элементов из $M$ равно нулю). Если бы $a$ принадлежало $A$, то для некоторого натурального $n$ имело бы место $a=z_{1}a_{1}+\dots + z_{n}a_{n}$ для некоторых целых чисел $z_{1},z_{2},\dots , z_{n}$ (причем $z_{n}$ не делится на $p$). Тогда $ag\left({b}\right)=0$ для некоторого многочлена $g\left({x}\right)$ степени $2n$ со старшим коэффициентом $z_{n}$ и без свободного члена. В простом кольце вычетов $Z_{p}$ имеет место $z\cdot z_{n}=1\left({mod p}\right)$ для некоторого целого $z$. Тогда имеем $zag\left({b}\right)=0$

Пусть $f\left({x}\right)$ - многочлен, полученный из многочлена $z\cdot \left({g\left({x}\right)}\right)$ заменой старшего коэффициента $zz_{n}$ на $1$. Нетрудно видеть, что $a\cdot f\left({b}\right)=zag\left({b}\right)=0$ и $f\left({x}\right)$ ---унитарный многочлен с целыми коэффициентами и без свободного члена. Это противоречит предположению. Вывод: $a\not\in A$. Тогда в силу финитной отделимости кольца $K$ должен найтись гомоморфизм $\varphi :K\rightarrow F$ для некоторого конечного кольца $F$ такой, что $\varphi \left({a}\right)\not\in \varphi \left({A}\right)$. Но в конечном кольце для любого элемента некоторая его степень является идемпотентом, то есть $\varphi \left({b^{2n}}\right)=\varphi \left({b^{n}}\right)$ для некоторого натурального числа $n$. Откуда $\varphi \left({a_{n}}\right)=\varphi \left({a}\right)$, то есть $\varphi \left({a}\right)\in \varphi \left({A}\right)$. Полученное противоречие доказывает лемму 1.
\end{proof}

\begin{lemma}
 Каждый примарный идеал в финитно отделимом конечно порожденном коммутативном кольце простой характеристики является простым идеалом или идеалом конечного индекса.
\end{lemma}
\begin{proof}
Пусть $I$ примарный идеал в финитно отделимом конечно порожденном коммутативном кольце $K=Z\left<a_{1},a_{2}, \dots , a_{n}\right>$ простой характеристики $p$. Предположим, что $K/I$ --- 		 бесконечное кольцо. Покажем, что фактор кольцо $K/I$ --- кольцо без делителей нуля. Предположим противное: $cd=0$ для некоторых отличных от нуля элементов $c,d$ из кольца $K/I$.

По определению примарного идеала $d^{n}=0$ в $K/I$ для некоторого натурального $n$. Поскольку $c\ne 0$ и $d\ne 0$, то $n>1$. Выберем $n>1$ наименьшим со свойством $d^{n}=0$. Применим лемму 1, взяв в качестве  $b$ какой-либо $\overline{a}_{i}$ ($\overline{a}_{i}$ --- образ $a_{i}$ при каноническом гомоморфизме $K\rightarrow K/I$), а в качестве $a$ элемент $d^{n-1}$. Тогда $a^{2}=0$, откуда $a\cdot f\left({b}\right)=0$ для некоторого унитарного многочлена  $f\left({x}\right)$ ( без свободного члена). Заметим, что $a\ne 0$. Поэтому по определению примарного идеала заключаем,что $f\left({b}\right)^{m}=0$ для некоторого натурального числа  $m$. Заметим, что многочлен $f\left({x}\right)$ можно рассматривать как ненулевой многочлен из $Z_{p}[x]$. Поэтому $f\left({x}\right)^{m}$ --- ненулевой многочлен. Заключаем, что $\overline{a}_{i}$ --- целый элемент кольца $K/I$ по определению. Мы получили, что кольцо $K/I$ порождается конечным набором целых элементов. Нетрудно видеть, что его аддитивная группа конечно порождена. Конечно порожденная группа простой характеристики является конечным множеством, то есть $K/I$ --- конечное кольцо. Это противоречит предположению. Противоречие получено из предположения, что кольцо $K/I$ не является кольцом без делителей нуля. Лемма 2 доказана.
\end{proof}

\begin{lemma}
В финитно отделимом кольце простой характеристики без делителей нуля любые два трансцендентных элемента является целозависимыми.
\end{lemma}
\begin{proof}
Пусть $K$ --- финитно отделимое кольцо простой характеристики  $p$ без делителей нуля. Пусть $a,b$ два произвольных трансцендентных элемента кольца $K$

По лемме 2 из работы \cite{kyblv} имеет место равенство
\begin{equation}
f_{0}\left({b}\right)a^{n}+f_{1}\left({b}\right)a^{n-1}+\dots + f_{n-1}\left({b}\right)a+f_{n-1}\left({b}\right)a=0 \tag{$*$}
\end{equation}
для некоторого натурального $n$ и некоторых многочленов $f_{i}\left({x}\right)\in Z_{p}[x]$ без свободных членов ($i=0, \dots ,n-1$), причем $f_{0}\left({x}\right)\in Z_{p}[x]$ --- ненулевой многочлен. Без ограничения общности считаем, что $n$ выбрано наименьшим из возможных. Это число $n$ обозначим через $D_{b}\left|{a}\right|$ (то есть через $D_{b}\left|{a}\right|$ обозначим наименьшую из возможных степеней ненулевых многочленов $\sigma _{b}\left({x}\right)$ (без свободного члена) с коэффициентами из кольца $Z\left<{b}\right>$, для которых $\sigma _{b}\left({a}\right)=0$ в кольце $K$. Такое число называют алгебраической степенью элемента $a$ относительно $b$ в кольце $K$.

Пусть $d\left({x}\right)=\gcd \left({f_{0}\left({x}\right),f_{1}\left({x}\right), \dots ,f_{n}\left({x}\right)}\right)$ и $k_{i}=f_{i}\left({x}\right):d\left({x}\right)$ для $i=0,1,n-1$. Ясно, что $k_{i}\in Z_{p}[x]$. Тогда по определению наибольшего общего делителя имеем $\gcd\left({k_{0},k_{1},\dots,k_{n-1}}\right)=1$. Покажем, что $k_{0}=1$. Предположим противное, то есть $k_{0}\ne 1$.

Пусть $g\left({x}\right)=xd\left({x}\right)$. Тогда $g\left({x}\right)$ --- многочлен без свободного члена. Кольцо $K$ можно считать модулем над кольцом многочленов $Z_{p}[x,y]$ с операцией $k\left({x,y}\right)\cdot c=k^{* }\left({b,a}\right)c+qc$, где $k\left({x,y}\right)$ --- произвольный многочлен из $Z_{p}[x,y]$, $c$ --- произвольный элемент кольца $K$, $q$ -свободный член многочлена $k\left({x,y}\right)$, а $k^{* }\left({x,y}\right)=k\left({x,y}\right)-q$.

Заметим, что из равенства (*) путем умножения на $b$ следует равенство:
\begin{equation*}
     g\left({b}\right)\cdot \left({k_{0}\cdot a^{n}+k_{1}\cdot a^{n-1}+\dots+k_{n-1}\cdot a}\right)=0.
\end{equation*}
Поскольку $b$ --- трансцендентный элемент и кольцо $K$ без делителей нуля, то из последнего равенства следуют равенства:
\begin{equation*}
    k_{0}\cdot a^{n}+k_{1}\cdot a^{n-1}+\dots+k_{n-1}\cdot a=0
\end{equation*}
    и
\begin{equation}
        D_{b}\left|{a}\right|=n.  \tag{**}
\end{equation}

Тогда для некоторого натурального $m$, не превосходящего $n-1$, многочлен $k_{m}$ не делится на некоторый неприводимый многочлен --- делитель многочлена $k_{0}$ в кольце многочленов $Z_{p}[x]$ (в противном случае нод многочленов $k_{i}$ был бы отличен от 1). Далее будем считать, что $m$ выбрано наименьшим из возможных с указанным свойством. Перепишем равенство (*) в двух видах:
\begin{gather}
     k_{0}\cdot a^{n}+\dots+k_{m-1}\cdot a^{n-m+1}=-k_{m}\cdot a^{n-m}-\dots-k_{n-1}\cdot a     \tag{1}\\
     k_{0}\cdot a^{n}=-k_{1}\cdot a^{n-1}-\dots-k_{n-1}\cdot a            \tag{2}
\end{gather}
Обозначим через $f\left({x,y}\right)$ и $\varphi \left({x,y}\right)$ следующие многочлены из  $Z_{p}[x,y]$:
\begin{gather*}
 f\left({x,y}\right)= k_{0}y^{n}+k_{1}y^{n-1}+ \dots + k_{m-1}y^{n-m+1},\\
 \varphi \left({x,y}\right)=-k_{m}y^{n-m} - \dots -k_{n-1}y.
\end{gather*}

Обозначим через $A$ подмножество кольца $K$ определяемое равенством:
\begin{equation*}
   A=Z_{p}[x]\cdot \left\{{k_{0}\cdot a}\right\}+Z_{p}[x]\cdot \left\{{k^{2}_{0}\cdot a^{2}}\right\}+\dots + Z_{p}[x]\cdot \left\{{k^{n-1}_{0}\cdot a^{n-1}}\right\}.
\end{equation*}
Ясно, что $A$ замкнуто по сложению (вычитанию), то есть $A$ --- аддитивная подгруппа кольца $K$. Из равенства $\left({2}\right)$ следует замкнутость $A$ по умножению. Чтобы в этом убедиться достаточно проверить, что \begin{equation*}
    \left({k_{0}\cdot a}\right)\left({k^{n-1}_{0}a^{n-1}}\right)\in A.
\end{equation*}  
Действительно,
\begin{equation*}
    \left({k_{0}\cdot a}\right)\cdot \left({k^{n-1}_{0}\cdot a^{n-1}}\right)=\left({k^{n-1}_{0}k_{0}}\right)\cdot a^{n}= k^{n-1}_{0}\cdot \left({-k_{1}\cdot a^{n-1}- \dots -k_{n-1}\cdot a}\right)\in A.
\end{equation*} 
Следовательно, $A$ --- подкольцо кольца $K$. Заметим, что $A$ содержит все элементы вида $k^{s}_{0}\cdot a^{s}$ для любого натурального $s$. Из определений вытекает, что $A$ --- подмодуль модуля $K$ над кольцом $Z_{p}[x]$. 

Обозначим через $M$ подмножество кольца $K$, определяемое следующим образом:
\begin{equation*}
     M=\{w=z_{n-1}\cdot a^{n-1}+ \dots + z_{i}\cdot a^{i} + \dots + z_{1}\cdot a \mid z_{i}\in Z_{p}[x], deg z_{i}<deg k^{i}_{0}\}
\end{equation*}
Отметим, что $M$ --- конечное множество. Пусть $M^{* }$ --- множество ненулевых элементов из $M$. Покажем, что $M^{* }\cap A=\left\{{0}\right\}$. Действительно, в противном случае получилось бы  $D_{b}\left|{a}\right|<n$, что противоречит $\left({**}\right)$. По условию кольцо $K$ финитно отделимо. Это означает, что существуют конечное кольцо $\Omega $ и гомоморфизм $\chi :K\rightarrow \Omega $ такой, что $\chi \left({M^{* }}\right)\cap \chi \left({A}\right)=\emptyset $.
Тогда в конечном кольце существуют натуральные числа $t$ и $h$ такие, что $\chi \left({k^{t}_{0}\cdot a}\right)=\chi \left({k^{t+h}_{0}\cdot a}\right)$. Из последнего равенства следует 
\begin{equation*}
    \chi \left({k^{t}_{0}\cdot a}\right)=\chi \left({k^{t+sh}_{0}\cdot a}\right)
\end{equation*} 
для любого натурального числа $s$, откуда 
\begin{equation*}
    \chi \left({k^{t}_{0}\cdot a^{s}}\right)=\chi \left({k^{t+sh}_{0}\cdot a^{s}}\right)=\chi \left({k^{t+sh-s}_{0}\cdot k_{0}a}\right)\chi \left({k^{s-1}_{0}\cdot a^{s-1}}\right).
\end{equation*}

Как отмечено выше, $k^{s}_{0}a^{s}\in A$ для всех $s$. Поэтому из последнего равенства вытекает 
\begin{equation}
    \chi \left({k^{t}_{0}a^{s}}\right)\in \chi \left({A}\right) \ \text{для всех } s\in \mathbb{N}
    \tag{3}
\end{equation}
Пусть $L=\{k_{0},k_{1},\dots,k_{m-1}\}$. Пусть $q$ --- некоторое натуральное число. Обозначим через $I\left({L^{q}}\right)$ идеал кольца $Z_{p}[x]$, порожденный множеством $L^{q}=L\cdot L\cdot L\cdots L$ ($q$ раз).

Для дальнейших рассуждений докажем следующую импликацию:\\
Если для некоторых $c_{i }\in I\left({L^{q}}\right)\left({i>n-m}\right)$ и для некоторого неотрицательно целого числа $z$ имеет место равенство \begin{equation*}
      \sum\limits_{i>n-m}{c_{i}\cdot a^{i}}=k^{z}_{m}\cdot \varphi \left({b,a}\right)+r\left({b,a}\right)
  \end{equation*} для некоторого многочлена $r\left({x,y}\right)$ без свободного члена из кольца $Z_{p}[x,y]$ степени меньшей $n-m$ относительно $y$, то имеет место аналогичное равенство
\begin{equation*}
     \sum\limits_{i>n-m}{c'_{i}\cdot a^{i}}=k^{z'}_{m}\cdot \varphi \left({b,a}\right)+r'\left({b,a}\right)
\end{equation*}
 для некоторых $c'_{i }\in I\left({L^{q+1}}\right)\left({i>n-m}\right)$ и для некоторого натурального числа $z'$ и некоторого многочлена  $r'\left({x,y}\right)$ из кольца $Z_{p}[x,y]$ степени меньшей $n-m$ относительно переменной $y$ без свободного члена.
 
Для доказательства этой импликации заметим следующее:\\
а) Для некоторого достаточно большого натурального числа $l$ в кольце многочленов  $Z_{p}[x,y]$ возможно деление с остатком, а именно:
\begin{equation*}
    k^{l}_{m}y^{i}=U_{i}\left({x,y}\right)\varphi \left({x,y}\right)+R_{i}\left({x,y}\right),  deg_{y}R_{i}\left({x,y}\right)<deg_{y}\varphi \left({x,y}\right)=n-m
\end{equation*}
для всех $i$ из промежутка $ n\ge i>n-m$, для которых $c_{i}\ne 0$. Заметим, что все многочлены $R_{i}\left({x,y}\right)$ без свободного члена.

б) Тогда из условия рассматриваемой импликации следует путем умножения на $k^{l}_{m}$:
\begin{equation*}
     \sum\limits_{i>n-m}{c_{i}k^{l}_{m}\cdot a^{i}}=k^{z+l}_{m}\cdot \varphi \left({b,a}\right)+k^{l}_{m}\cdot r\left({b,a}\right),
\end{equation*}
 откуда на основании (а) получаем
\begin{equation*}
     \sum\limits_{i>n-m}{c_{i}\cdot \left({U_{i}\left({x,y}\right)\cdot \varphi \left({b,a}\right)+R_{i}\left({b,a}\right)}\right)}=k^{z+l}_{m}\cdot \varphi \left({b,a}\right)+k^{l}_{m}\cdot r\left({b,a}\right).
\end{equation*}

в) Из равенства $\left({1}\right)$ следует $\varphi \left({b,a}\right)=f\left({b,a}\right)$. Поэтому заменяя в предыдущем равенстве $\varphi \left({b,a}\right)$ на $f\left({b,a}\right)$, получаем
\begin{equation*}
    \sum\limits_{i>n-m}{c_{i}\left({U_{i}\left({x,y}\right)f\left({b,a}\right)+R_{i}\left({b,a}\right)}\right)}=k^{z+l}_{m}\varphi \left({b,a}\right)+k^{l}_{m}r\left({b,a}\right).
\end{equation*} 
Коэффициенты многочлена $f\left({x,y}\right)$ относительно переменной $y$ принадлежат множеству $L$ по определению. Поэтому коэффициенты многочлена $c_{i}\left({U_{i}\left({x,y}\right)f\left({x,y}\right)}\right)$ относительно переменной $y$ принадлежат множеству $Z\left({L^{q+1}}\right)$, поскольку по условию все $c_{i }\in I\left({L^{q}}\right)$.	Чтобы доказать рассматриваемую импликацию, осталось в левой части последнего равенства раскрыть скобки и перенести все $c_{i}R_{i}\left({b,a}\right)$ в правую часть. Итак, импликация доказана.

Из этой импликации методом математической индукции приходим к выводу: для любого натурального числа $s$, и для некоторых целых чисел $c_{i }\in I\left({L^{s}}\right)\left({i>n-m}\right)$, и для некоторого целого неотрицательного числа $z$ имеет место равенство
\begin{equation*}
   \sum\limits_{i>n-m}{c_{i}a^{i}}=k^{z}_{m}\varphi \left({b,a}\right)+r\left({b,a}\right)
\end{equation*} 
для некоторого многочлена с $r\left({x,y}\right)$ из $Z_{p}[x,y]$ степени меньшей $n-m$ и без свободного члена. Базой индукции является равенство $\left({1}\right)$.

Заметим, что 
\begin{equation*}
    k^{z}_{m}\varphi \left({x,y}\right)+r\left({x,y}\right)=q_{1}y+q_{2}y^{2}+\dots+q_{n-m}y^{n-m}
\end{equation*} 
для некоторых многочленов $q_{j}$ ($j=1,\dots,n-m$) из кольца $Z_{p}[x]$, причем $q_{n-m}=-k^{z+1}_{m}$.
Поделим каждое $q_{j}$ с остатком на $k^{j}_{0}$ в кольце многочленов $Z_{p}[x]$ и получим равенства:
\begin{center}
    $q_{j}=p_{j}k^{j}_{0}+q^{* }_{j}$, где $0\le q^{* }_{j}<k^{j}_{0}$ ($j=1\dots,n-m$).
\end{center} 
Заметим, что $q^{* }_{n-m}\ne 0$, поскольку $k^{z+1}_{m}$ не делится на $k_{0}$. Имеем 
\begin{equation*}
    k^{z}_{m}\cdot \varphi \left({b,a}\right)+r\left({b,a}\right)=\sum\limits_{j=1}^{n-m}{\left({p_{j}k^{j}_{0}+q^{* }_{j}}\right)\cdot a}^{j}=\sum\limits_{j=1}^{n-m}{p_{j}k^{j}_{0}a}^{j}+\sum\limits_{j=1}^{n-m}{q^{* }_{j}a}^{j}.
\end{equation*} 
Получаем 
\begin{equation*}
    \chi \left({\sum\limits_{i>n-m}{c_{i}a^{i}}}\right)=\chi \left({k^{z}_{m}\varphi \left({b,a}\right)+r\left({b,a}\right)}\right).
    \end{equation*}
 Из предыдущего равенства получаем:
    \begin{equation}
 \chi \left({\sum\limits_{j=1}^{n-m}{p_{j}k^{j}_{0}\cdot a}^{j}}\right)+\chi \left({\sum\limits_{j=1}^{n-m}{q^{* }_{j}\cdot a}^{j}}\right). \tag{4}
\end{equation}

Для достаточно большого числа $s$ любой элемент $c_{i}$ из идеала $I\left({L^{s}}\right)$ делится на $k^{t}_{0}$. Поэтому на основании $\left({3}\right)$ заключаем $\chi \left({\sum\limits_{i>n-m}{c_{i}a^{i}}}\right)\in \chi \left({A}\right)$.

На основании определения множества $A$ имеем $\left({\sum\limits_{j=1}^{n-m}{p_{j}k^{j}_{0}a}^{j}}\right)\in A$. Поэтому из равенства $\left({4}\right)$ получаем $\chi \left({\sum\limits_{j=1}^{n-m}{q^{* }_{j}a}^{j}}\right)\in \chi \left({A}\right)$. На основании определения множества $M$ имеем $\sum\limits_{j=1}^{n-m}{q^{* }_{j}a}^{j}\in M$. Поскольку $\chi \left({M^{* }}\right)\cap \chi \left({A}\right)=\emptyset $, то заключаем $\sum\limits_{j=1}^{n-m}{q^{* }_{j}a}^{j}=0$. Но тогда получилось бы  $D_{b}\left|{a}\right|\le n-m<n$, что противоречит выбору $n$.
Противоречие получилось из предположения $k_{0}\ne 1$. Это означает, что $k_{0}=1$.

Итак, мы доказали равенство $a^{n}+k_{1}\cdot a^{n-1}+ \dots + k_{n-1}\cdot a=0$. Аналогично доказывается равенство $b^{n'}+l_{1}\cdot b^{n'-1}+ \dots + l_{n'-1}\cdot b=0$ для некоторого натурального $n'$ и некоторых многочленов $l_{1},l_{2},\dots, l_{n'-1}$ из $Z_{p}[y]$. Без ограничения общности, можно считать, что $n$ больше максимальной из степеней многочленов $l_{1},l_{2},\dots,l_{n'-1}$, а $n'$ больше максимальной из степеней многочленов $k_{1},k_{2},\dots,k_{n}$. Тогда 
\begin{equation*}
    \varphi \left({x,y}\right)=y^{n}+k_{1}y^{n-1}+\dots + k_{n-1}y+x^{n'}+l_{1}x^{n'-1}+ \dots + l_{n'-1}x
\end{equation*} 
является унитарным многочленом и $\varphi \left({b,a}\right)=0$. Вывод: элементы $a$ и $b$ целозависимы. Лемма 3 доказана.
\end{proof}

\begin{lemma}
Если конечно порожденное коммутативное кольцо $K$ финитно отделимо, то оно является конечным расширением некоторого своего идеала кручения $I_{k}$, который является прямым произведением конечного набора колец различных простых характеристик.
\end{lemma}

\begin{proof}
 Пусть $K$ --- конечно порожденное коммутативное финитно отделимое кольцо. Пусть $\{a_{1},\dots,a_{n}\}$ --- порождающее множество кольца $K$. По леммам 3 и  4 из работы \cite{kyblv} каждый элемент финитно отделимого кольца имеет конечное целое кручение свободное от квадратов. Отсюда следует $k_{i}\cdot f_{i}\left({a_{i}}\right)=0$ для некоторого набора унитарных многочленов $f_{i}\left({x}\right)$ с целыми коэффициентами без свободных членов и некоторого набора натуральных чисел $k_{i}$ свободных от квадратов ($i=1, \dots , n$). 
 
 Пусть 
 \begin{center}
     $k=$ Н.О.К $\left({k_{1}, \dots ,k_{n}}\right)$, $I_{k}=\{b\in K |kb=0\}$.
 \end{center} В фактор-кольце $K/I_{k}$ имеет место $f_{i}\left({\overline{a}_{i}}\right)=0$, где $\overline{a}_{i}$ --- образ элемента $a_{i}$ при каноническом гомоморфизме $K\rightarrow K/I _{k} $($i=1,2, \dots ,n$). Это означает, что все $\overline{a}_{i}$ являются целыми алгебраическими элементами в кольце $K/I _{k}$. Отсюда следует, что аддитивная группа кольца $K/I _{k}$ конечно порождена. Поскольку все числа $k_{1}, \dots ,k_{n}$ --- свободны от квадратов, то и  $k=$ Н.О.К $ \left({k_{1},\dots,k_{n}}\right)$ свободно от квадратов.

Пусть $k=p_{1}\cdot p_{2}\cdot\dots\cdot p_{n}$ --- разложение на простые множители.
Из определений вытекает: 
\begin{center}
    Н.О.Д $\left({k/p_{1},\dots, k/p_{n}}\right)=1$.
\end{center} 
По теореме Евклида о линейном выражении наибольшего общего делителя существует семейство целых чисел $z_{i}$ ($i=1, \dots ,n$) такое, что
\begin{equation}
  z_{1}\left({k/p_{1}}\right)+\dots+z_{n}\left({k/p_{n}}\right)=1.  \tag{5}
\end{equation}
Рассмотрим семейство идеалов $\left({k/p_{i}}\right)\cdot I_{k}$ кольца $I_{k}$  Из равенства (5) следует:
\begin{center}
    $I_{k}=$ $z_{1}\left({k/p_{1}}\right)I_{k}+\dots+z_{n}\left({k/p_{n}}\right)I_{k}$.
\end{center}

Каждое кольцо $\left({k/p_{i}}\right)\cdot I_{k}$ --- это кольцо простой характеристики $p_{i}$. Поэтому сумма идеалов различных простых характеристик является прямой суммой. Осталось заметить, что прямая сумма конечного набора идеалов изоморфна их прямому произведению. Лемма 4 доказана.
\end{proof}

\begin{corollary}
 Утверждения леммы 4 остаются справедливыми для конечно порожденных PI-колец благодаря теореме Ширшова о высоте \cite{shir}.
\end{corollary}

\begin{lemma}
Если конечно порожденное коммутативное кольцо простой характеристики финитно отделимо, то оно является подпрямым произведением конечного кольца и некоторого конечного набора колец без делителей нуля.
\end{lemma}

\begin{proof}
Пусть конечно порожденное коммутативное кольцо $K$ простой характеристики является финитно отделимым. По теореме Гильберта любое конечно порожденное коммутативное кольцо является нетеровым. По теореме Ласкера--Нетер \cite{zarc} каждый идеал кольца $K$ является пересечением некоторого конечного семейства примарных идеалов. В частности, $\bigcap\limits_{i=1}^{n}{I_{i}}=0$ для некоторого семейства примарных идеалов $I_{i}$ кольца $K$.
Это означает, что кольцо $K$ раскладывается в подпрямое произведение колец $K/I_{i}$. Осталось применить лемму 2, из которой следует, что каждое кольцо $K/I_{i}$ либо конечно, либо является кольцом без делителей нуля. Лемма 5 доказана.
\end{proof}

\begin{lemma}
Конечно порожденное коммутативное кольцо без делителей нуля простой характеристики, в котором каждые два трансцендентных элемента целозависимы, является финитно отделимым.
\end{lemma}

\begin{proof}
Пусть $K$ конечно порожденное коммутативное кольцо без делителей нуля имеет простую характеристику. Пусть любые два трансцендентных элемента кольца $K$ целозависимы.

Покажем, что $K$ --- финитно отделимое кольцо. Пусть $a\notin A$ для некоторого элемента $a$ и подкольца $A$ кольца $K$. Если в кольце $A$ нет трансцендентных элементов, то каждый элемент $b$ из $A$ является регулярным, то есть представимым в виде $b=b^{2}b'$ для некоторого элемента $b'\in A$. По умножению множество $A^{* }$, ненулевых элементов $A$, образует коммутативную регулярную полугруппу с сокращением. Как хорошо известно, $A^{* }$ --- это группа. Поэтому $A$ --- поле. Любое конечно порожденное коммутативное кольцо является финитно аппроксимируемым кольцом \cite{orz},\cite{kybl}. Поэтому $K$ --- финитно аппроксимируемое кольцо. Поэтому и поле $A$ --- финитно аппроксимируемое кольцо. Откуда заключаем, что $A$ --- конечное поле. Тогда возможность разделения элемента $a$ от конечного подкольца следует из финитной аппроксимируемости кольца $K$.

Если в кольце $A$ есть некоторый трансцендентный элемент $b$, то все образующие кольца $K$ являются либо целыми, либо целыми относительно $b$ по лемме 3. Отсюда вытекает, что кольцо $K$ является конечно порожденным модулем над моногенным подкольцом $Z\left<{b}\right>$.

По предложению 5 из работы [8] следует, что кольце $K$ каждый элемент $a$ можно финитно отделить от любого подкольца, содержащего $b$, но не содержащего $a$. Таковым является и подкольцо $A$. Лемма 6 доказана.
\end{proof}

\begin{lemma}
Кольцо многочленов $C_{p}[x]$ от одного переменного с коэффициентами из произвольного конечного поля $C_{p}$ финитно отделимо.
\end{lemma}

\begin{proof}

Пусть $f\left({x}\right)\notin A$, где $f\left({x}\right)$ --- некоторый многочлен из кольца $C_{p}[x]$, а $A$ --- некоторое подкольцо кольца $C_{p}[x]$. Если $A$ состоит из многочленов нулевой степени, то   $A$ --- конечное множество. Возможность финитного отделения любого элемента $f\left({x}\right)\notin A$ от конечного подкольца $A$ следует из финитной аппроксимируемости конечно порожденных коммутативных колец \cite{orz}, \cite{kybl}.

Пусть в кольце $A$ есть некоторый многочлен $b=g\left({x}\right)$ и пусть степень многочлена $g\left({x}\right)$ равна $n$ и $n\ge 1$. Рассмотрим кольцо $C_{p}[x]$ как модуль над моногенным подкольцом $Z<b>$.

Покажем, что этот модуль является конечно порожденным и конечное множество $B=\{h\left({x}\right)|degh\left({x}\right)\le n\}$, состоящее из всех многочленов $h\left({x}\right)$ из $C_{p}[x]$ степени меньшей или равной $n$, является порождающим множеством этого модуля.

Пусть $[B]$ --- подмодуль модуля $C_{p}[x]$, как модуля над кольцом $Z<b>$, порожденный множеством $B$. Покажем, что $C_{p}[x]=[B]$, то есть каждый многочлен $\alpha \left({x}\right)$ из $C_{p}[x] \in [B]$.
Если $deg\alpha \left({x}\right)\le n$, то $\alpha \left({x}\right)\in B\subset [B]$. Пусть $deg\alpha \left({x}\right)>n$. Дальнейшие рассуждения проводим по индукции. Предположим, что для любого многочленам $\beta \left({x}\right)$ из $C_{p}[x]$, степень которого меньше степени многочлена $\alpha \left({x}\right)$ имеет место  $\beta \left({x}\right)\in [B]$.
Разделим многочлен $\alpha \left({x}\right)$ на многочлен $g\left({x}\right)$ в кольце $C_{p}[x]$ с остатком, то есть рассмотрим равенство:
\begin{center}
 $\alpha \left({x}\right)=g\left({x}\right)\cdot \beta \left({x}\right)+r\left({x}\right)$
\end{center}
для некоторых многочленов $\beta \left({x}\right)$ и $r\left({x}\right)$ из  $C_{p}[x]$ и $degr\left({x}\right)<degg\left({x}\right)$. Тогда по определению $r\left({x}\right)\in B\subset [B]$. Но из сравнения степеней многочленов следует, что
\begin{center}
    $deg\alpha \left({x}\right)=degg\left({x}\right)+deg\beta \left({x}\right)$.
\end{center}
Откуда 
\begin{center}
    $deg\beta \left({x}\right)=deg\alpha \left({x}\right)-degg\left({x}\right)=deg\alpha \left({x}\right)-n<deg\alpha \left({x}\right)$.
\end{center}
По индукционному предположению $\beta \left({x}\right)\in [B]$. Но тогда
\begin{center}
    $g\left({x}\right)\cdot \beta \left({x}\right)=b\cdot \beta \left({x}\right)\in b[B]\subset [B]$.
\end{center}
Мы получили, что \begin{center}
    $g\left({x}\right)\cdot \beta \left({x}\right)\in [B]$ и $r\left({x}\right)\in [B]$.
\end{center} 
Поэтому $\alpha \left({x}\right)=g\left({x}\right)\cdot \beta \left({x}\right)+r\left({x}\right)\in [B]$.

Итак, мы доказали по индукции, что $C_{p}[x]=[B]$. Заметим, что $B$ --- конечное множество.Это означает, что $C_{p}[x]$ является конечно порожденным модулем над своим подкольцом $Z<b>$. Поскольку $b\in A$, то по предложению 5 из \cite{kyblv} существуют конечное кольцо $F$ и гомоморфизм $\gamma :C_{p}[x]\rightarrow F$ такой, что $\gamma \left({a}\right)\notin \gamma \left({A}\right)$. Вывод: кольцо $C_{p}[x]$ финитно отделимо. Лемма 7 доказана.
\end{proof}

\section{Основные результаты}

В этом параграфе мы излагаем доказательства основных результатов.

Основным результатом настоящей работы является следующая теорема.
\begin{proclaim}{Теорема} Для того, чтобы конечно порожденное коммутативное кольцо было финитно отделимо, необходимо и достаточно, чтобы оно было конечным расширением некоторого идеала кручения $I_{k}$ ($k$ --- число свободное от квадратов), который является  подпрямым произведением конечного кольца и некоторого конечного набора колец без делителей нуля простых характеристик, в которых любые два трансцендентные элемента целозависимы.
\end{proclaim}

\begin{proof}
Необходимость. Следует из {лемм 2--5} и замкнутости класса финитно отделимых колец относительно подколец и гомоморфных образов.

Достаточность. Следует из {леммы 6} настоящей работы, {предложения 2} работы \cite{kyblv} и замкнутости класса финитно отделимых колец относительно конечных прямых произведений и подколец.
\end{proof}

\begin{proclaim}{Следствие}
Конечно порожденное коммутативное кольцо является финитно отделимым в том и только в том случае, когда каждое его двупорожденное подкольцо финитно отделимо.
\end{proclaim}

\begin{proof}
 Вытекает из теоремы и замкнутости класса финитно отделимых колец относительно подколец и гомоморфных образов.
\end{proof}

\begin{proclaim}{Предложение} Двупорожденное кольцо простой характеристики $K=Z_{p}\left<{a,b\mid f\left({a,b}\right)=0}\right>$, где $f\left({x,y}\right)$ --- однородный многочлен над полем $Z_{p}$ от двух переменных, будет финитно отделимым, тогда и только тогда, когда $f\left({x,y}\right)$ --- сепарабельный многочлен, то есть является произведением различных неприводимых многочленов.
\end{proclaim}

\begin{proof}
Необходимость. Пусть кольцо $K=\mathbb{Z}_{p}\left<{a,b\mid f\left({a,b}\right)=0}\right>$ финитно отделимо, где
 $f\left({x,y}\right)$ --- однородный многочлен от двух переменных над полем $\mathbb{Z}_{p}$ некоторой степени $n\ge 1$. Предположим, что многочлен $f\left({x,y}\right)$ не является сепарабельным, Тогда\begin{center}
      $f\left({x,y}\right)=g\left({x,y}\right)^{2}h\left({x,y}\right)$
 \end{center} для некоторых однородных многочленов $g\left({x,y}\right)$ и $h\left({x,y}\right)$ ненулевой степени из  $\mathbb{Z}_{p}[x,y]$.
 
 Пусть $I=\{u\in K\mid uh\left({a,b}\right)=0\}$. Ясно, что $I$ --- идеал кольца $K$ и в факторе $K/I$ выполнено $\overline{g\left({a,b}\right)}^{2}=0$, где чертой обозначается образ элемента при каноническом гомоморфизме из $K\rightarrow K/I$. Поскольку класс финитно отделимых колец замкнут относительно гомоморфизмов, то применяя {лемму 1} к кольцу $K/I$, получим $\overline{g\left({a,b}\right)\varphi \left({b}\right)}=0$ для некоторого унитарного многочлена $\varphi \left({x}\right)$ из $Z_{p}$ без свободного члена. Это означает, что $g\left({a,b}\right)\varphi \left({b}\right)\in I$ в кольце $K$, откуда следует
\begin{center}
 $g\left({a,b}\right)\varphi \left({b}\right)h\left({a,b}\right)=0$.
\end{center}
  Из последнего равенства следует
  \begin{center}
       $g\left({x,y}\right)\varphi \left({y}\right)h\left({x,y}\right)=f\left({x,y}\right)\psi \left({x,y}\right)$
  \end{center} 
  для некоторого многочлена $\psi \left({x,y}\right)$ из $Z_{p}[x,y]$. Получаем равенство в кольце многочленов $Z_{p}[x,y]$:
  \begin{center}
       $g\left({x,y}\right)\varphi \left({y}\right)h\left({x,y}\right)=g\left({x,y}\right)^{2}h\left({x,y}\right)\psi \left({x,y}\right)$.
  \end{center}
Из последнего равенства путем сокращения получаем равенство  
\begin{center}
    $\varphi \left({y}\right)=g\left({x,y}\right)\psi \left({x,y}\right)$,
\end{center} 
что невозможно. Противоречие получено из предположения, что многочлен $f\left({x,y}\right)$ не является сепарабельным. Вывод: $f\left({x,y}\right)$ --- сепарабельный многочлен. Необходимость доказана.

Достаточность. Пусть двупорожденное кольцо $K$ простой характеристики $p$ задано соотношением $K=Z_{p}\left<{a,b\mid f\left({a,b}\right)=0}\right>$, где $f\left({x,y}\right)$ --- однородный сепарабельный многочлен от двух переменных над полем $Z_{p}$ некоторой степени $n\ge 1$. Однородный сепарабельный многочлен раскладывается однозначно в произведение различных неприводимых однородных многочленов 
\begin{center}
    $f\left({x,y}\right)=f_{1}\left({x,y}\right)\cdot f_{2}\left({x,y}\right)\cdots f_{m}\left({x,y}\right)$.
\end{center} 
Главный идеал $I\left({f}\right)$ кольца многочленов $Z_{p}[x,y]$, порожденный многочленом $f=f\left({x,y}\right)$, является пересечением семейства главных идеалов $I\left({f_{i}}\right)$, порожденных многочленами $f_{i}=f_{i}\left({x,y}\right)$, то есть
\begin{equation*}
    I\left({I}\right)=\overset{i=m}{\underset{i=1}{\cap }}I\left({f_{i}}\right).
\end{equation*}  
Откуда следует, что кольцо $K=Z^{* }_{p}[x,y]/I\left({f}\right)$ является подпрямым произведением семейства колец $Z^{* }_{p}[x,y]/I\left({f_{i}}\right)$. Финитная отделимость кольца $K$ следовала бы из финитной отделимости каждого из колец $Z^{* }_{p}[x,y]/I\left({f_{i}}\right)$ в силу замкнутости класса финитно отделимых колец относительно конечных прямых произведений, гомоморфизмов и подколец. Поэтому далее, без ограничения общности, будем считать, что $f\left({x,y}\right)$ --- неприводимый и унитарный относительно $x$ многочлен.

Это означает, что кольцо $Z^{* }_{p}[x,y]/I\left({f}\right)$ --- является кольцом без делителей нуля. Поэтому $K$ является кольцом без делителей нуля  (поскольку $K$ изоморфно кольцу $Z^{* }_{p}[x,y]/I\left({f}\right)$). Пусть $\overline{K}$ --- поле частных кольца $K$ и пусть $\Lambda $ --- алгебраическое замыкание поля $\overline{K}$. Кольцо $K$ является подкольцом поля $\Lambda $. Пусть поле $\overline{Z_{p}}$ --- алгебраическое замыкание простого поля $Z_{p}$. Поле  $\overline{Z_{p}}$ можно рассматривать как подполе $\Lambda $.
Над полем $\overline{Z_{p}}$ унитарный многочлен  $f\left({t,1}\right)$ раскладывается в произведение линейных множителей 
\begin{equation*}
    f\left({t,1}\right)=\left({t-\lambda _{1}}\right)\left({t-\lambda _{2}}\right)..\left({t-\lambda _{n}}\right)
\end{equation*} 
для некоторых $\lambda _{1},\lambda _{2}, \dots, \lambda _{n}$ из $\overline{Z_{p}}$ (см. \cite{byrbv}). Поэтому 
\begin{multline*}
    f\left({x,y}\right)=y^{n}\left({\frac{x}{y}-\lambda _{1}}\right)\left({\frac{x}{y}-\lambda _{2}}\right)\dots\left({\frac{x}{y}-\lambda _{n}}\right) ={}\\{}= \left({x-\lambda _{1}y}\right)\left({x-\lambda _{2}y}\right)\dots\left({x-\lambda _{n}y}\right).
\end{multline*}
 Поэтому в поле  $\Lambda $ имеет место равенство:
 \begin{equation*}
    f\left({a,b}\right)=\left({a-\lambda _{1}b}\right)\left({a-\lambda _{2}b}\right)\cdots\left({a-\lambda _{n}b}\right).
\end{equation*} 
Поскольку $f\left({a,b}\right)=0$, то и $\left({a-\lambda _{1}b}\right)\left({a-\lambda _{2}b}\right)\cdots\left({a-\lambda _{n}b}\right)=0$. Это означает, что $a-\lambda _{i}b=0$ для некоторого $i$, откуда $a=\lambda _{i}b$. Пусть $C_{p}$ --- подполе поля $\overline{Z_{p}}$, порожденное элементом $\lambda _{i}$. Из определений следует, что $\lambda _{i}$ --- целый элемент над $Z_{p}$, то есть $\varphi \left({\lambda _{i}}\right)=0$, для некоторого ненулевого многочлена $\varphi \left({x}\right)$ из кольца. $Z_{p}[x]$. Поэтому $C_{p}$ совпадает с подкольцом, порожденным $\lambda _{i}$, которое, в свою очередь, порождается как абелева группа конечным набором элементов $1,\lambda _{i},\lambda ^{2}_{i},\dots, \lambda ^{m}_{i}$, где $m$ --- это степень многочлена $\varphi \left({x}\right)$. Конечно порожденная абелева группа простой характеристики является конечной. Вывод: $C_{p}$ --- конечное поле. Заметим, что элемент $b$ является трансцендентным над полем $C_{p}$, то есть $g\left({b}\right)\ne 0$ для всех ненулевых многочленов $g\left({x}\right)$ из кольца $C_{p}[x]$. В противном случае элемент $b$ был бы целым элементом над кольцом $Z_{p}$, то есть $h\left({b}\right)=0$ для некоторого унитарного многочлена $h\left({x}\right)$ без свободного члена из $Z_{p}[x]$ (поскольку каждый элемент кольца $C_{p}$ является целым над $Z_{p}$). Но в кольце $K$ равенство $h\left({b}\right)=0$ означает, что многочлен $h\left({x}\right)$ делится на многочлен $f\left({x,y}\right)$ в кольце многочленов $Z_{p}[x,y]$. Последнее, очевидно, неверно. Полученное означает, что в поле $\Lambda $ подкольцо $K^{* }$, порожденное элементом $\lambda _{i}$ и элементом $b$, изоморфно кольцу многочленов $C_{p}[x]$. По {лемме 7} кольцо $C_{p}[x]$ финитно отделимо. Итак, кольцо $K^{* }$ финитно отделимо. Остается заметить, что кольцо $K$ является подкольцом кольца $K^{* }$, поскольку $a=\lambda _{i}b\in K^{* }$. Предложение доказано полностью.
\end{proof}

 Есть достаточно простой критерий определяющий сепарабельность многочлена от одного переменного и от двух переменных над полем.
\begin{proclaim}{Замечание.} Многочлен ненулевой степени от одного переменного $\varphi \left({t}\right)$ является сепарабельным в том и только в том случае, когда этот многочлен взаимно прост со своей производной $\varphi '\left({t}\right)$. (Это несложное упражнение).
\end{proclaim}

 Как нетрудно видеть, однородный многочлен от двух переменных $f\left({x,y}\right)$ ненулевой степени из $Z_{p}[x,y]$ будет сепарабельным тогда и только тогда, когда сепарабельным будет многочлен от одного переменного $\varphi \left({t}\right)=f\left({t,1}\right)$.

\section{Примеры и контрпримеры}

\begin{proclaim}{Гипотеза}
Конечно порожденное коммутативное кольцо простой характеристики без делителей нуля будет финитно отделимым, тогда и только тогда, когда ее любые два трансцендентные образующие являются целозависимыми.
\end{proclaim}
Следующий пример показывает, что данная гипотеза в такой формулировке, вообще говоря, неверна.

\begin{proclaim}{Пример 1}
 Кольцо $K=Z_{3}\left<{a,b\mid a^{2}+b-b^{2}=0}\right>$ не является финитно отделимым.
\end{proclaim}

\begin{proof}
Во-первых, нетрудно убедиться, что многочлен $f\left({x,y}\right)=x^{2}+y-y^{2}$ является неприводимым (то есть не разлагается на множители ненулевой степени) в кольце $Z_{3}[x,y]$. Поэтому фактор-кольцо $Z_{3}[x,y]/I\left({x^{2}+y-y^{2}}\right)$, где $I\left({x^{2}+y-y^{2}}\right)$ --- главный идеал, порожденный многочленом $x^{2}+y-y^{2}$, является кольцом без делителей нуля.

Поскольку кольцо $K$ вкладывается естественным образом в кольцо $Z_{3}[x,y]/I\left({x^{2}+y-y^{2}}\right)$, то можно сделать вывод о том, что $K$ --- кольцо без делителей нуля.

Пусть $c=a-b$, откуда $a=c+b$. Тогда
\begin{equation*}
    c^{2}+2bc+b^{2}=b^{2}-b,
\end{equation*}  
откуда $c^{2}+2bc+b=0$. Из последнего равенства получаем
\begin{equation*}
    \left({2c+1}\right)b+c^{2}=0.
\end{equation*}

Предположим, что $b\in Z\left<c\right>$. Это означает, что $b=f\left({c}\right)$ для некоторого ненулевого многочлена $f\left({x}\right)$ без свободного члена из кольца многочленов $Z_{3}[x]$.

Если степень многочлена $f\left({x}\right)$ равна 1, то есть $f\left({x}\right)=zx$ для некоторого числа $z=1$ или $z=2$, то в кольце $K$ имеет место 
\begin{equation*}
    \left({2c+1}\right)zc+c^{2}=0,
\end{equation*} 
откуда 
\begin{equation*}
    \left({2z+1}\right)c^{2}+zc=0.
\end{equation*}
Если $z=1$, то получим $c=0$, что противоречит выбору $c$. Если же $z=2$, то получаем $2c^{2}+2c=0$, откуда $c^{2}+c=0$. Это означает, что $c$ --- целый алгебраический элемент.

Если степень многочлена $f\left({x}\right)$ больше $1$, то из равенства 
\begin{equation*}
    \left({2c+1}\right)f\left({c}\right)+c^{2}=0
\end{equation*} 
также следует, что $c$ --- целый алгебраический элемент. Тогда аддитивная группа кольца  $Z\left<c\right>$ конечно порождена и имеет характеристику 3. Заключаем, что $Z\left<c\right>$  --- конечное кольцо. Поскольку по предположению $b=f\left({c}\right)\in Z\left<c\right>$, заключаем, что $Z\left<b\right>$ --- конечное кольцо (как подкольцо конечного).

 Тогда $g\left({b}\right)=0$ для некоторого ненулевого многочлена $g\left({x}\right)$ без свободного члена из кольца многочленов $Z_{3}[x]$. Из изоморфизма колец $K$ и $Z^{* }_{3}[x,y ]/\left({x^{2}+y-y^{2}}\right)$ заключаем, что многочлен $g\left({x}\right)$ делится на многочлен $x^{2}+y-y^{2}$ в кольце многочленов $Z_{3}[x,y]$. А это неверно! Противоречие получено из предположения $b\in Z\left<c\right>$.
 
Вывод: $b\notin Z\left<c\right>$. Предположим, что кольцо $K$ является финитно отделимым кольцом. Тогда $\varphi \left({b}\right)\notin \varphi \left(Z\left<c\right>\right)$ для некоторого гомоморфизма $\varphi:K\rightarrow F$ из кольца $K$ в некоторое конечное кольцо $F$. Тогда $\varphi \left({b^{2n}}\right)=\varphi \left({b^{n}}\right)$ для некоторого натурального $n$. Отсюда 
\begin{equation*}
    \varphi \left({\left({2c+1}\right)^{2n-1}b^{2n}}\right)=\varphi \left({\left({2c+1}\right)^{2n-1}b^{n}}\right).
\end{equation*}
Из последнего равенства следует 
\begin{equation*}
    \varphi \left({\left({\left({2c+1}\right)b}\right)^{2n-1}b}\right)=\varphi \left({\left({2c+1}\right)^{n-1}\left({\left({2c+1}\right)b}\right)^{n}}\right),
\end{equation*}
откуда
\begin{equation*}
     \varphi \left({\left({-c^{2}}\right)^{2n-1}b}\right)=\varphi \left({\left({2c+1}\right)^{n-1}\left({-c^{2}}\right)^{n}}\right)\in \varphi \left({Z<c>}\right).
\end{equation*}
Имеем 
\begin{center}
    $\varphi \left({\left({-c^{2}}\right)^{2n-1}b}\right)\in \varphi \left({Z<c>}\right)$ и $\varphi \left({\left({2c+1}\right)b}\right)=\varphi \left({-c^{2}}\right)\in \varphi \left({Z<c>}\right)$.
\end{center}

Поскольку многочлены $u=\left({-x^{2}}\right)^{2n-1}$ и $v=2x+1$ являются взаимно простыми в кольце многочленов $Z_{3}[x]$, то их Н.О.Д$\left({u,v}\right)=1$. Тогда по теореме о линейном выражении Н.О.Д имеет место равенство
\begin{equation*}
    \alpha \left({x}\right)u\left({x}\right)+\beta \left({x}\right)v\left({x}\right)=1
\end{equation*} 
в кольце многочленов $Z_{3}[x]$. Тогда получаем 
\begin{center}
    $\varphi \left({\alpha \left({c}\right)\left({-c^{2}}\right)^{2n-1}b}\right)\in \varphi \left({Z\left<c\right>}\right)$ и $\varphi \left({\beta \left({c}\right)\left({2c+1}\right)b}\right)\in \varphi \left({Z\left<c\right> }\right)$,
\end{center}
откуда
\begin{equation*}
     \varphi \left({\alpha \left({c}\right)\left({-c^{2}}\right)^{2n-1}b+\beta \left({c}\right)\left({2c+1}\right)b}\right)\in \varphi \left({Z\left<c\right>}\right),
\end{equation*}
 то есть
\begin{equation*}
    \varphi \left({\left({\alpha \left({c}\right)\left({-c^{2}}\right)^{2n-1}+\beta \left({c}\right)\left({2c+1}\right)}\right)b}\right)\in \varphi \left({Z\left<c\right>}\right).
\end{equation*}
  Из последнего равенства получаем $\varphi \left({b}\right)\in \varphi \left(Z\left<c\right>\right)$, что противоречит выбору $\varphi$.

Противоречие получено из предположения о том, что кольцо $K$ является финитно отделимым кольцом. Вывод: $K$ не является финитно отделимым кольцом. Что и требовалось доказать.
\end{proof}

Пример 1 показывает, что из целозависимости образующих кольца, вообще говоря, не следует целозависимость произвольной пары трансцендентных элементов.

\begin{proclaim}{Пример 2}
Кольцо $K=Z_{2}\left<{a,b\mid a^{2}+b-b^{2}=0}\right>$ является финитно отделимым.
\end{proclaim}
\begin{proof}

Во-первых, нетрудно убедиться, что многочлен $f\left({x,y}\right)=x^{2}+y-y^{2}$ является неприводимым (то есть не разлагается на множители ненулевой степени) в кольце $Z_{2}[x,y]$.

Поэтому фактор-кольцо $Z_{2}[x,y]/I\left({x^{2}+y-y^{2}}\right)$, где $I\left({x^{2}+y-y^{2}}\right)$ -- главный идеал, порожденный многочленом $x^{2}+y-y^{2}$, является кольцом без делителей нуля.

Поскольку кольцо $K$ вкладывается естественным образом в кольцо $Z_{2}[x,y]/I\left({x^{2}+y-y^{2}}\right)$, то можно сделать вывод о том, что $K$ --- кольцо без делителей нуля.

Пусть $\varphi \left({x}\right)=x^{2}-x$. Заметим, что каждый элемент $c$ из кольца $K$ можно записать в виде $c=f\left({b}\right)a+g\left({b}\right)$ для некоторых многочленов $f\left({x}\right),g\left({x}\right)$ из кольца многочленов $Z_{2}[x]$, причем  $g\left({x}\right)$ без свободного члена.

Пусть $c$ --- ненулевой элемент, тогда хотя бы один из многочленов $f\left({x}\right),g\left({x}\right)$ не нулевой. Тогда $c-g\left({b}\right)=f\left({b}\right)a$, откуда 
\begin{equation*}
    \left({c-g\left({b}\right)}\right)^{2}=f\left({b}\right)^{2}a^{2}=-f\left({b}\right)^{2}\varphi \left({b}\right),
\end{equation*}
откуда \begin{equation*}
    c^{2}+g\left({b}\right)^{2}+f\left({b}\right)^{2}\varphi \left({b}\right)=0.
\end{equation*}
Если $g\left({x}\right)^{2}+f\left({x}\right)^{2}\varphi \left({x}\right)$ ненулевой многочлен, то $b$ --- целый алгебраический над кольцом $Z\left<c\right>$.

Но $a$ --- целый алгебраический над кольцом $Z\left<b\right>$,  поскольку в кольце $K$ по условию выполняется равенство $a^{2}+\varphi \left({b}\right)=0$. Тогда можно сделать вывод, что и $a$ --- целый алгебраический над кольцом $Z\left<c\right>$.

Вывод: кольцо $K$ простой характеристики $2$ является конечно порожденным модулем над моногенным подкольцом $Z\left<{c}\right>$. По {предложению 5} из работы \cite{kyblv} это кольцо $K$ финитно отделимо от подколец, содержащих $c$. Поскольку в качестве $c$ может выступать любой ненулевой элемент кольца $K$, можно сделать вывод, что в этом случае кольцо $K$ финитно отделимо.

Если $g\left({x}\right)^{2}+f\left({x}\right)^{2}\varphi \left({x}\right)$ нулевой многочлен, то $g\left({x}\right)^{2}+f\left({x}\right)^{2}\left({x^{2}-x}\right)$ --- нулевой многочлен в $Z_{2}[x]$. Тогда $xf\left({x}\right)^{2}=g\left({x}\right)^{2}+f\left({x}\right)^{2}x^{2}$. Слева стоит многочлен нечетной степени, а справа многочлен четной степени. Равенство таких многочленов невозможно в кольце $Z_{2}[x]$, поэтому рассматриваемый подслучай не имеет места.

Вывод: кольцо $K$ финитно отделимо. Что и требовалось доказать.
\end{proof}

\section{Заключение. Открытые вопросы.}

Полученные в настоящей работе результаты показывают, что вопрос об описании конечно порожденных коммутативных колец со свойством финитной отделимости сводится к описанию финитно отделимых колец простой характеристики без делителей нуля. Вопрос об описании последних связан с описанием унитарных неприводимых многочленов от двух переменных с коэффициентами из простого поля, определяющих финитно отделимые двупорожденные кольца. В этой проблематике представляют интерес следующие вопросы.

\begin{proclaim}{Вопрос 1}
Описать финитно отделимые двупорожденные кольца простой характеристики, заданные одним определяющим соотношением.
\end{proclaim}

\begin{proclaim}{Вопрос 2}
Описать финитно отделимые простые квадратичные расширения моногенных колец простой характеристики, то есть кольца вида $K=Z_{p}\left<{a,b\mid a^{2}=f\left({b}\right)}\right>$, где $p$ --- простое число, $f\left({x}\right)$ --- многочлен без свободного члена из кольца многочленов $Z_{p}[x]$.
\end{proclaim}

\end{document}